\documentclass[11pt]{article}
\usepackage{amsmath,amsfonts, amssymb,amsthm,multicol,graphicx,tikz, enumitem}
\usepackage{makeidx}
\usepackage{mleftright}
\usetikzlibrary{calc}

\DeclareMathOperator{\Var}{Var}

\DeclareMathOperator{\I}{(i)}
\DeclareMathOperator{\ii}{(ii)}
\DeclareMathOperator{\iii}{(iii)}

\newcommand{\ZZ}{\mathbb{Z}}
\newcommand{\RR}{\mathbb{R}} 
\newcommand{\NN}{\mathbb{N}} 

\newcommand{\EE}{\mathbb{E}}

\newtheorem{prop}{Proposition}   
\newtheorem{lem}[prop]{Lemma} 
\newtheorem{thm}[prop]{Theorem} 

\title{Concentration of the largest induced tree size of $G_{n,p}$ around the standard expectation threshold}
\author{Jakob Hofstad\footnote{ Institute for Computer Science, Heidelberg University, Germany. Email: \\hofstad@informatik.uni-heidelberg.de}}

\begin{document}

\maketitle

\begin{center}
    {\bf Abstract}
\end{center}

\noindent Let $T(G)$ be the size of the largest induced tree of $G$, and let $G_{n,p}$ be the binomial random graph. Kamaldinov, Skorkin, and Zhukovskii proved that $T(G_{n,p})$ equals one of two consecutive values with high probability if $p$ is constant, and more recently, Oropeza extended this result to include all vanishing $p$ such that $p > n^{-\frac{e-2}{3e-2} + \epsilon}$, where $e$ is Euler's constant. We further extend this result to all vanishing $p$ such that $p \gg n^{-1/2} \ln^{3/2} n$, and furthermore, we show that, for $p$ such that $n^{-1} \ll p \ll n^{-1/2}, \ T(G_{n,p})$ cannot be concentrated at the standard expectation threshold.

\section{Introduction}

Define $T(G)$ to be the maximum size of an induced tree of $G$. Let $G_{n,p}$ be the binomial random graph; i.e., a graph on $n$ vertices such that each pair of vertices is joined by an edge with probability $p$, and all $\binom{n}{2}$ of these events are independent. We say that an event $\mathcal{E}_n$ (dependent on $n$) happens ``with high probability", or ``whp", if $\lim_{n \to \infty} \Pr[\mathcal{E}_n] = 1$. Finally, we say that a variable $X$ defined on $G_{n,p}$ is ``concentrated on
$k = k(n)$ values" if there is a function $f(n)$ such that $X \in \{f(n), \dots ,f(n)+k-1\}$ whp.

The study of $T(G_{n,p})$ began with Erd\H{o}s and Palka, who proved that, if $p$ is a non-trivial constant, $T(G_{n,p}) \approx 2 \log_{1/(1-p)}n$ whp \cite{EZ-first-forest-result}. Kamaldinov, Skorkin, and Zhukovskii later sharpened this result: they proved that, for constant $p$, $T(G_{n,p})$ is concentrated on two values~\cite{KSZ-induced-sparse-graphs}. Oropeza then proved that two-point concentration of $T(G_{n,p})$ also holds for all functions $p = p(n)$ such that $p = o(1)$ and $p > n^{-\frac{e-2}{3e-2} + \epsilon}$ for any fixed $\epsilon > 0$ \cite{Oropeza-conc-of-induced-tree-prev}.

Closely related to the size of the maximum induced tree $T(G)$ is the independence number $\alpha(G)$, at least with regard to their behavior in $G_{n,p}$. 
The nature of $\alpha(G_{n,p})$ has been long studied; for constant $p$, two-point concentration was established independently by Erd\H{o}s and Bollob\'{a}s \cite{BE-Cliques} and by Matula \cite{Matula-cliques} in the 1970's, at the same asymptotic value of $2 \log_{1/(1-p)}n$. In 1990, Frieze \cite{Frieze-indep-asymptotic} used a combination of the second moment method and martingale inequalities to show that, if $p = o(1)$ and $p = \omega(1/n)$, then whp

\begin{equation*}
    \alpha(G_{n,p}) = \frac{2 \ln(np) + 2 + o(1)}{p}.
\end{equation*}

Much more recently, in a series of two papers \cite{BH-2point-Gnm, BH-2point-Gnp}, Bohman and the author proved that $\alpha(G_{n,p})$ is concentrated on two values if $p = \omega(n^{-2/3} \ln^{2/3} n)$, and moreover, that this order of $p$ is tight. Unlike the Erd\H{o}s-Bollob\'{a}s and Matula papers, these newer results required the use of a random variable that counts not independent sets themselves, but sets which are ``almost" independent but contain many independent sets within them.

Also closely related to $T(G)$ is
the maximum induced {\it forest} size $F(G)$; the study of $F(G_{n,p})$ has many close parallels with that of $T(G_{n,p})$ and $\alpha(G_{n,p})$.
Krivoshapko and Zhukovskii proved that, for constant $p$ ($p \not= 0$), $F(G_{n,p})$ is concentrated on the {\it same} two values as $T(G_{n,p})$ \cite{KZ-forests-p-constant}. Akhmejavona and Kozhevnikov \cite{AK-max-induced-forest} used techniques similar to the earlier result of Frieze to show that, if $p = o(1)$ and $p = \omega(1/n)$, then whp

\begin{equation*}
    F(G_{n,p}) = \frac{2 \ln(np) + 2 + o(1)}{p}.
\end{equation*}

In this paper, we prove that 2-point concentration of $T(G_{n,p})$ holds for a wide range of values of $p$. Defining $X_k$ as the number of induced graphs of size $k$ in $G_{n,p}$ which are trees, we also show that, for slightly smaller $p$, $T(G_{n,p})$ {\it cannot} be concentrated around the value of $k$ for which $\EE[X_k]$ transitions from being $\omega(1)$ to being $o(1)$, known as the ``expectation threshold".

\begin{thm} \label{thm: main}

There exists an integer function $k_0 = k_0(n,p)$ such that, for all $p = p(n)$ such that $\frac{1}{n} \ll p \ll \frac{1}{\ln^2 n}$:

\begin{enumerate}
    \item[$\I$] The largest value of $k$ such that $\EE[X_k]$ is at least 1 is either $k_0$ or $k_0 + 1$.

    \item[$\ii$] If $p = \omega(n^{-1/2} \ln^{3/2} n)$, then $T(G_{n,p}) \in \{k_0, k_0 + 1\}$ whp.

    \item[$\iii$] If $p = o(n^{-1/2})$, then $T(G_{n,p}) < k_0 - 1/(4np^2)$ whp.
\end{enumerate}
    
\end{thm}

The theorem is proved by using two other random variables $Y_k$ and $W_k$. Both count induced trees of size $k$ but with restrictions: $Y_k$ counts the induced trees of size $k$ where vertices outside the tree have at least three neighbors inside the tree, while $W_k$ counts the induced trees of size $k$ which are maximal induced trees. The first and second moment methods with $Y_k$ are used to prove part (ii) of the theorem, and the first moment of $Y_k$ is used for part (iii).

The structure of the paper is as follows. Part (i) and the upper bound for $T(G_{n,p})$ in part (ii) will be proved in Section \ref{sec: 1st moment}. The lower bound for $T(G_{n,p})$ in part (ii) will be proved in Section \ref{sec: 2nd moment}. Part (iii) will be proved in Section \ref{sec: drift}. Section \ref{sec: closing remarks} contains closing remarks. An appendix contains two technical lemmas used in Section \ref{sec: 2nd moment}.

\subsection{Key definitions and inequality}

We close this section by defining some key notation and citing a strong form of Stirling's Inequality with the gamma function $\Gamma$. For a positive integer $n$, let $[n]$ the the set $\{1, 2, \dots, n\}$. For a graph $G$ and set $A \subset V(G)$, let $G[A]$ be the induced graph of $G$ on vertex set $A$. For a vertex $v$ in a graph $G$, let $\mathcal{N}(v)$ the set of the neighbors of $v$ in $G$.

A strong version of Stirling's Inequality is as follows (see (2.2) and (3.9) in \cite{Artin-gamma-funct}): for all positive real $n$,

    \begin{equation}
        \sqrt{2\pi} n^{n+1/2} e^{-n} \leq \Gamma(n+1) \leq \sqrt{2\pi} n^{n+1/2} e^{-n+1/(12n)}. \label{equ: Stirling's}
    \end{equation}

We mostly use the inequality above in combination with the fact that $\Gamma(n+1) = n!$ for all $n \in \NN$ (see (2.6) in \cite{Artin-gamma-funct}), though it will also be used for half-integer values of $n$.

\section{First moment calculation} \label{sec: 1st moment}

Throughout the rest of the paper, $\frac{1}{n} \ll p \ll \frac{1}{\ln^2 n}$ always.

Recall that $X_k$ is the number of induced subgraphs with size $k$ that are trees. The expected value of $X_k$ in $G_{n,p}$ is

\begin{equation} \label{equ: E[X]}
    \EE[X_k] = \binom{n}{k} k^{k-2} p^{k-1} (1-p)^{\binom{k}{2} - k +1}.
\end{equation}

We shall define

\begin{equation*}
    k_0 := \max\{k: \EE[X_k] > \ln(np)\};
\end{equation*}

this is certainly well-defined since $\EE[X_1] = n$ and $k \leq n$. If $k = \Theta(1)$ then (\ref{equ: E[X]}) will be $\Omega(n)$, and if $k = \Theta(n)$ (suppose $k > \varepsilon n$ for some $\varepsilon \in (0,1)$) then

\begin{equation*}
    \EE[X_k] \leq 2^n (\varepsilon np)^{\varepsilon n} e^{-pk(k-3)/2} \leq (2 (\varepsilon np)^{\varepsilon} e^{-\varepsilon^2 np + 3\varepsilon p})^n = o(1).
\end{equation*}

Therefore, it is sufficient to just consider $k$ which is both $\omega(1)$ and $o(n)$. Using (\ref{equ: Stirling's}), taking the $k$th root of (\ref{equ: E[X]}) gives

\begin{equation*}
    (\EE[X_k])^{1/k} \approx enp (1-p)^{k/2}.
\end{equation*}

This expression will be $\Theta(1)$ if and only if

\begin{equation*}
    k = 2\log_{1/(1-p)}(enp) + o(p^{-1});
\end{equation*}

Since $p = o(\ln^{-2} n)$, the above is equivalent to

\begin{equation}
    k = \frac{2 \ln(np) + 2 + o(1)}{p} \label{equ: k approx}
\end{equation}

and likewise equivalent to

\begin{equation}
    (1-p)^{k/2} \approx \exp\{-kp/2\} \approx \frac{1}{enp}. \label{equ: k usefuler approx}
\end{equation}

Moreover, a larger $k$ will make $\EE[X_k]$ be $o(1)$, and a smaller $k$ will make $\EE[X_k]$ be $\Omega(n)$, therefore $k_0$ must satisfy (\ref{equ: k approx}) and (\ref{equ: k usefuler approx}). For such a range of $k$,

\begin{equation} \label{equ: ratio over k}
\frac{\EE[X_{k+1}]}{\EE[X_k]} = (n-k) \left(\frac{k+1}{k}\right)^{k-2}p (1-p)^{k-1} \approx \frac{1}{enp},
\end{equation}

therefore $\EE[X_{k}] = o(1)$ for all $k \geq k_0 + 2$.\\

From this final statement, it follows that $k_0$ satisfies part (i) of Theorem \ref{thm: main}. Also, by Markov's Inequality on $X_{k_0+2}$, $T(G_{n,p})$ is at most $k_0+1$ whp, which is precisely the upper bound given in part (ii) of Theorem \ref{thm: main}.\\

\section{Second moment calculation} \label{sec: 2nd moment}

The goal of this Section is to verify the lower bound of part (ii) in Theorem \ref{thm: main}; i.e., to show that $T(G_{n,p}) \geq  k_0$ whp if $p = \omega(n^{-1/2} \ln^{3/2} n)$; hence, throughout this Section we assume that $p = \omega(n^{-1/2} \ln^{3/2} n)$.\\

We prove that $T(G_{n,p}) \geq k_0$ whp using the second moment method on the random variable $Y_k$ that counts the number of induced trees of size $k$ such that all vertices not in $Y$ have at least 3 neighbors in $Y$. We have (using (\ref{equ: k approx}))

\begin{align*}
    \frac{\EE[Y_k]}{\EE[X_k]}&=\left(1 - \sum_{i=0}^{2} \binom{k}{\ell}p^i(1-p)^{k-\ell}\right)^{n-k}\\&=
    (1-O(k^2/n^2))^{n-k} \\&=
    1 - O\left(\frac{k^2}{n}\right) \\&=
    1 - o(1).
\end{align*}

Since $\EE[X_{k_0}] = \omega(1)$ by definition, then $\EE[Y_{k_0}] = \omega(1)$ also. Therefore, to prove that $T(G_{n,p}) \geq k_0$ whp, by Chebyshev's Inequality, it is enough to prove the following Lemma, which we will do in the remainder of this Section:\\

\begin{lem} \label{lem: 2nd moment}

For $k$ satisfying (\ref{equ: k approx}) and (\ref{equ: k usefuler approx}) such that $\EE[X_k] = \omega(1)$:
    \begin{equation*}
    \frac{\Var[Y_k]}{\EE[Y_k]^2} = o(1).
\end{equation*}
\end{lem}

For ease of notation, we fix one such $k$ and write $X_k$ as $X$ and $Y_k$ as $Y$.

We start by expressing both $X$ and $Y$ as sums of indicator variables. For any set $A$ with $|A| = k$, and if $T_A$ is a tree on the vertices of $A$, define

\begin{align*}
    X_{A, T_A} &= \begin{cases}
        1 \text{ if } G_{n,p}[A] = T_A\\
        0 \text{ otherwise}
    \end{cases} \\
    Y_{A, T_A} &= \begin{cases}
        1 \text{ if } G_{n,p}[A] = T_A \text{ AND } \forall v \not\in A, \mathcal{N}(v) \cap A \geq 3\\
        0 \text{ otherwise}
    \end{cases}
\end{align*}

Let $\mathcal{S}$ be the set of all possible tuples $(A, T_A)$ for which $A \in \binom{[n]}{k}$ ad $T_A$ is a tree with vertex set $A$. Clearly, $\EE[X] = |\mathcal{S}|\cdot \EE[X_{A,T_A}]$ and $\EE[Y] = |\mathcal{S}| \cdot\EE[Y_{A,T_A}]$ for all $(A,T_A) \in \mathcal{S}$. Furthermore, since $\EE[Y] \approx \EE[X]$ (and hence $\EE[Y_{A,T_A}] \approx \EE[X_{A,T_A}]$), to prove Lemma \ref{lem: 2nd moment} it is enough to show that

\begin{equation} \label{equ: full covariance}
    \frac{1}{\EE[X]^2} \sum_{(A, T_A),(B, T_B) \in \mathcal{S}}  (\EE[Y_{A,T_A}Y_{B,T_B}] - \EE[Y_{A,T_A}]\EE[Y_{B,T_B}]) = o(1).
\end{equation}

As is standard with the second moment method, we will split the sum above based on the size of $|A \cap B|$. We will first consider the case where $|A \cap B| = 0$, the only case for which the $ - \EE[Y_{A,T_A}]\EE[Y_{B,T_B}]$ in the summand is necessary to use. We make use of the fact that, if $|A \cap B| = 0$, then $X_{A,T_A}$ and $X_{B,T_B}$ are independent events:

\begin{align*}
    \frac{1}{\EE[X]^2} &\sum_{\substack{(A, T_A),(B, T_B) \in \mathcal{S}:\\
    |A \cap B| = 0}}  (\EE[Y_{A,T_A}Y_{B,T_B}] - \EE[Y_{A,T_A}]\EE[Y_{B,T_B}]) \\=
    \frac{1}{|\mathcal{S}|^2}&\sum_{\substack{(A, T_A),(B, T_B) \in \mathcal{S}:\\
    |A \cap B| = 0}}  \frac{\EE[Y_{A,T_A}Y_{B,T_B}] - \EE[Y_{A,T_A}]\EE[Y_{B,T_B}]}{\EE[X_{A,T_A}]^2} \\ \leq
    \frac{1}{|\mathcal{S}|^2}&\sum_{\substack{(A, T_A),(B, T_B) \in \mathcal{S}:\\
    |A \cap B| = 0}}  \frac{\EE[X_{A,T_A}X_{B,T_B}] - (1-o(1))\EE[X_{A,T_A}]\EE[X_{B,T_B}]}{\EE[X_{A,T_A}]^2} \\=
   \frac{1}{|\mathcal{S}|^2}&\sum_{\substack{(A, T_A),(B, T_B) \in \mathcal{S}:\\
    |A \cap B| = 0}}  o(1) \\= o(1)  &.
\end{align*}

Furthermore, the sum above with the restriction $|A \cap B| = k$ instead is at most $\frac{\EE[Y]}{\EE[X]^2} = o(1)$, so it is enough to consider the case $|A \cap B| \in [k-1]$. Next, we define parameters to equal partial sums of (\ref{equ: full covariance}) dependent not only on a parameter $\ell := |A \cap B|$, but also on a parameter $m$, equal to the number of components of $T_A$ (likewise $T_B$) induced on $A \cap B$ (if they are not the same, then surely $\EE[Y_{A,T_A} Y_{B,T_B}] = 0$), which we simply write as $T_A \cap T_B$. Define

\begin{align}
x_{\ell,m} &:= \frac{1}{\EE[X]^2} \sum_{\substack{(A, T_A),(B, T_B) \in \mathcal{S}:\\
    |A \cap B| = \ell, \\
    T_A \cap T_B\text{ has $m$ components}}} \EE[X_{A,T_A}X_{B,T_B} ] \label{equ: x def} \\
y_{\ell,m} &:= \frac{1}{\EE[X]^2} \sum_{\substack{(A, T_A),(B, T_B) \in \mathcal{S}:\\
    |A \cap B| = \ell, \\
    T_A \cap T_B\text{ has $m$ components}}}  \EE[Y_{A,T_A}Y_{B,T_B} ] \label{equ: y def}
\end{align}

To prove Lemma \ref{lem: 2nd moment}, it is enough to show that

\begin{equation} \label{equ: sum of ys}
    \sum_{\ell=1}^{k-1} \sum_{m=1}^{\ell} y_{\ell,m} = o(1).
\end{equation}

Clearly $y_{\ell,m} \leq x_{\ell,m}$ by definition of $X_{A,T_A}$ and $Y_{A,T_A}$; the following lemma demonstrates where $y_{\ell,m}$ significantly differs from $x_{\ell,m}$:

\begin{lem} \label{lem: y less than x}
Let $\ell \in [k-1]$ and $m \in [\ell]$. If $(k-\ell)p < 1$ and $2(k-\ell) - m + 1 > 0$, then

\begin{equation*}
    y_{\ell,m} \leq x_{\ell,m}\cdot((k-\ell)p)^{2(k-\ell)-m+1}.
\end{equation*}

\begin{proof}
    Suppose that we have fixed sets $A,B$ and trees $T_A, T_B$ (which coincide on $A \cap B$) such that $|A \cap B| = \ell$ and there are $m$ components in $T_A \cap T_B$. We will condition $G_{n,p}$ on the event $\mathcal{E} := \{X_{A,T_A} X_{B,T_B} = 1\}$. Note that $\mathcal{E}$ only depends on the edges within $\binom{A}{2} \cup \binom{B}{2}$, so all pairs of vertices in $(A \backslash B, B \backslash A)$ appear as edges in $G_{n,p}$ (conditioned on $\mathcal{E}$) independently with probability $p$.\\

    Next, suppose that the vertices of $A \backslash B$ are $a_1, \dots, a_{k-\ell}$, and $D(i) := \mathcal{N}(a_i) \cap A \cap B$. Because $T_A[A \cap B]$ has $m$ components, there are at most $(k-\ell)+m-1$ edges from $A \backslash B$ to $A \cap B$, so $\sum_{i=1}^{k-\ell}D(i) \leq k-\ell+m-1$. If $Y_{B,T_B} = 1$, then each vertex $a_{i} \in A \backslash B$ has
    at least $\min\{3-D(i),0\}$ neighbors in $B \backslash A$ which, by union bound, happens with probability at most $\min\{((k-\ell)p)^{3 - D(i)},1\}$. Furthermore, these $k-\ell$ events are all independent of each other and independent of $\mathcal{E}$. Therefore:

\begin{align*}
    \Pr[Y_{A,T_A} Y_{B,T_B} = 1 \mid \mathcal{E}] &\leq \Pr[Y_{B,T_B} = 1 \mid \mathcal{E}] \\&\leq
    \prod_{i=1}^{k-\ell} \min\{((k-\ell)p)^{3 - D(i)},1\} \\&\leq
    \min\left\{\prod_{i=1}^{k-\ell} ((k-\ell)p)^{3-D(i)}, 1\right\} \\&\leq
    \min\{((k-\ell)p)^{3(k-\ell) - (k-\ell+m-1)},1\} \\&=
    ((k-\ell)p)^{2(k-\ell)-m+1}.
\end{align*}

Therefore $\Pr[Y_{A,T_A} Y_{B,T_B} = 1] \leq ((k-\ell)p)^{2(k-\ell)-m+1} \Pr[X_{A,T_A}X_{B,T_B} = 1]$, and since $(A,T_A)$ and $(B,T_B)$ were arbitrary such that $|A \cap B|=\ell$ and $T_A \cap T_B$ has $m$ components, then by (\ref{equ: x def}) and (\ref{equ: y def}) the Lemma holds.

\end{proof}

\end{lem}

\subsection{A general upper bound for $x_{\ell,m}$}

First, for fixed sets $A$ and $B$, we determine the number of possible trees $T_A$ and $T_B$ that coincide on $A \cap B$ such that the number of components in this overlap is $m$. We count in the following way: suppose that the vertex sets of the components of $T_A \cap T_B$ are $F_1, \dots, F_m$ with $f_{i} := |F_{i}|$ for all $i \in [m]$. Then:

\begin{enumerate}
    \item The total number of possibilities for $T_A \cap T_B$ is
        \begin{equation} \label{equ: trees}
        \prod_{i=1}^{m} f_i^{f_i - 2}. 
        \end{equation}

    \item The number of ways to complete the trees $T_A$ and $T_B$, by Lemma \ref{lem: joining forests} in the appendix, is 

        \begin{equation} \label{equ: completing trees} \left(\left(\prod_{i=1}^{m}f_i\right) k^{k-\ell-1} (k-\ell)^{m-1}\right)^2.
\end{equation}
\end{enumerate}

Therefore, by multiplying the expressions (\ref{equ: trees}) and (\ref{equ: completing trees}) with $\binom{\ell}{f_1, \dots, f_m}$, the number of ways to arrange $F_1, \dots, F_m$ in $A \cap B$, then enumerating over all possible $f_1, \dots, f_m$, then finally dividing by $m!$ to eliminate overcounting by permuting the labels of $F_1,\dots, F_{m}$, we have, for fixed $A$, $B$ with $|A \cap B| = \ell$:

\begin{align*}
&|T_A, T_B:
T_A \cap T_B\text{ has $m$ components}| \\=& \sum_{\substack{ f_1, \dots, f_m \in \ZZ^{+}:\\ f_1 + \dots + f_m = \ell}} \frac{\ell!}{m!}\left(\prod_{i=1}^{m} \frac{f_i^{f_i}}{f_i!}\right) k^{2k-2\ell-2} (k-\ell)^{2m-2}.
\end{align*}

Next, we use Lemma \ref{lem: hyperball} from the appendix to give a useful upper bound of the above expression:

\begin{align*}
&|T_A, T_B:
T_A \cap T_B\text{ has $m$ components}| \\&\leq
\left(\frac{e^{\ell} \ell^{(m-2)/2}}{2^{m/2}\Gamma(m/2)} \right)\frac{\ell!}{m!} k^{2k-2\ell-2} (k-\ell)^{2m-2} \\&=
k^{2k-2\ell-2} \left(\ell^{(m-2)/2} \ell! e^{\ell}\right)(k-\ell)^{2m-2} \left(\frac{1}{2^{m/2} \Gamma(m/2)m!}\right).
\end{align*}

Next, we multiply the above expression by $\binom{n}{k} \binom{k}{\ell} \binom{n-k}{k-\ell}$, the number of ways to choose sets $A$ and $B$ such that $|A \cap B| = \ell$:

\begin{align*}
    &|(A,T_A), (B,T_B) \in \mathcal{S}: |A \cap B| = \ell, \text{ and $T_A, T_B$ coincide on $A \cap B$ with $m$ components}| \\\leq&
    \binom{n}{k}\binom{n-k}{k-\ell} \frac{k!}{(k-\ell)!} k^{2k-2\ell-2} \left(\ell^{(m-2)/2} e^{\ell}\right)(k-\ell)^{2m-2} \left(\frac{1}{2^{m/2} \Gamma(m/2)m!}\right).
\end{align*}

Next, suppose that $T_A$ and $T_B$ coincide on $A \cap B$; then 
there are $2k-\ell+m-2$ edges in $T_A \cup T_B$, hence 
\begin{equation} \label{equ: edge probability}
    \EE[X_{A,T_A} X_{B,T_B}] = \left(\frac{p}{1-p}\right)^{2k-\ell+m-2} (1-p)^{2\binom{k}{2}-\binom{\ell}{2}}.
\end{equation}

Putting everything together with (\ref{equ: x def}) and recalling (\ref{equ: E[X]}): 

\begin{equation} \label{equ: x_{ell,m} upper bound}
    x_{\ell,m} \leq \frac{k! \binom{n-k}{k-\ell}}{(k-\ell)! \binom{n}{k}} k^{-2\ell+2}\left(\ell^{(m-2)/2} e^{\ell}\right)(k-\ell)^{2m-2} \left(\frac{1}{2^{m/2} \Gamma(m/2)m!}\right) \left(\frac{p}{1-p}\right)^{-\ell+m} (1-p)^{-\binom{\ell}{2}}.
\end{equation}

With a general upper bound of $x_{\ell,m}$, together with Lemma \ref{lem: y less than x}, we now work to bound the sum (\ref{equ: sum of ys}). We consider two cases, depending on whether $\ell$ is less or more than $k/2$.

\subsection{Case 1: $\ell \leq k/2$}

Here we show

\begin{equation*}
    \sum_{\ell=1}^{\lfloor k/2 \rfloor} \sum_{m=1}^{\ell} x_{\ell,m} = o(1);
\end{equation*}
indeed, it is sufficient to just use the parameter $x_{\ell,m}$ instead of $y_{\ell,m}$ for this range of $\ell$.

First, note that 

\begin{equation*}
    \binom{n-k}{k-\ell} \leq \binom{n}{k-\ell} = \binom{n}{k} \frac{k(k-1) \dots (k-\ell+1)}{(n-k+1)(n-k+2) \dots (n-k+\ell)} \leq \binom{n}{k} \left(\frac{k}{n-k}\right)^{\ell}.
\end{equation*}

Therefore, by (\ref{equ: x_{ell,m} upper bound}):

\begin{align*}
    x_{\ell,m} &\leq \left(\frac{3 e^{\ell p/2}}{np}\right)^\ell \left(\frac{1}{2^{m/2} \Gamma(m/2)m!}\right) \left(2k^2p\ell^{0.5}\right)^{m}\\&=
    \left(\Theta(1)\frac{e^{\ell p/2}}{np}\right)^\ell \left(\Theta(1)\frac{k^2p\ell^{0.5}}{m^{1.5}}\right)^{m} \qquad \text{by (\ref{equ: Stirling's}).}
\end{align*}

Since, for all $\ell < k/2$ and $m < \ell$,  $\left(k^2p\ell^{0.5}m^{-1.5}\right)^{m}$ increases by a multiplicative factor of $\Omega(kp) \gg 1$ when $m$ increases by 1, then the sum is dominated by the case $m=\ell$:

\begin{align*}
    \sum_{\ell=1}^{\lfloor k/2 \rfloor} \sum_{m=1}^{\ell} x_{\ell,m} &= \sum_{\ell=1}^{\lfloor k/2 \rfloor} \left(\Theta(1)\frac{e^{\ell p/2}}{np}\right)^\ell \left(\Theta(1)\frac{k^2p\ell^{0.5}}{\ell^{1.5}}\right)^{\ell} \\&=
    \sum_{\ell=1}^{\lfloor k/2 \rfloor} \left(\Theta(1)\frac{e^{\ell p/2}}{\ell}\frac{k^2}{n}\right)^{\ell}.
\end{align*}

Since (by (\ref{equ: k usefuler approx}))

\begin{align*}
    \ell < \frac{1}{p} &\Longrightarrow \frac{e^{\ell p/2}}{\ell} = O(1/\ell) = O(1)
    \\ \ell \in [1/p, k/2] &\Longrightarrow \frac{e^{\ell p/2}}{\ell} = O(e^{kp/4}p) = O(n^{0.5}p^{1.5}),
\end{align*}

then (by (\ref{equ: k approx}))

\begin{align*}
    \sum_{\ell=1}^{\lfloor k/2 \rfloor} \sum_{m=1}^{\ell} x_{\ell,m} &= 
    \sum_{\ell=1}^{\lfloor k/2 \rfloor} \left(O(1 + n^{0.5} p^{1.5})\frac{k^2}{n}\right)^{\ell} \\&=
    \sum_{\ell=1}^{\lfloor k/2 \rfloor} \left(O\left(\frac{k^2}{n} + \frac{k^2p^{1.5}}{n^{0.5}} \right)\right)^{\ell}\\&=
    o(1).
\end{align*}

\subsection{Case 2: $\ell > k/2$}

Let $j := k-\ell$ for simplicity. We show

\begin{equation*}
    \sum_{j=1}^{\lfloor k/2 \rfloor} \sum_{m=1}^{k-j} y_{k-j,m} = o(1).
\end{equation*}

By (\ref{equ: x_{ell,m} upper bound}) (and recalling (\ref{equ: E[X]}) again),

\begin{align}
    x_{k-j,m} &\leq \frac{1}{\EE[X]} \frac{k! \binom{n-k}{j}}{j!} k^{2j-k}\left(k^{(m-2)/2} e^{k-j}\right)j^{2m-2} \left(\frac{1}{2^{m/2} \Gamma(m/2)m!}\right) \left(\frac{p}{1-p}\right)^{j+m-1} (1-p)^{\binom{k}{2}-\binom{k-j}{2}} \nonumber \\&=
    \frac{1}{\EE[X]}\left(\frac{\binom{n-k}{j}}{j!}(1-p)^{(k-j)j + \binom{j}{2}}\right)\left(\frac{k^{2} p}{e(1-p)} \right)^j \left(\frac{1}{2^{m/2} \Gamma(m/2)m!}\right) \left(\frac{k^{0.5}j^2 p}{1-p}\right)^{m-1} \frac{k!e^k}{k^{k+0.5}} \nonumber \\&\leq
     \left(\frac{1}{(j!)^2}(1-p)^{(k-j)j + \binom{j}{2}}\right)\left(\frac{n k^{2} p}{e(1-p)} \right)^j \left(\frac{1}{2^{m/2} \Gamma(m/2)m!}\right) \left(\frac{k^{0.5}j^2 p}{1-p}\right)^{m-1} \frac{k!e^k}{k^{k+0.5}} \nonumber \\&=
      \left(\Theta(1)\frac{n k^{2} p (1-p)^{k} e^{jp/2}}{j^2} \right)^j \left(\Theta(1)\frac{k^{0.5}j^2 p}{m^{1.5}}\right)^{m-1} 
    \qquad \text{by (\ref{equ: Stirling's})} \nonumber \\&=
      \left(\Theta(1)\frac{k^{2} e^{jp/2}}{npj^2} \right)^j \left(\Theta(1)\frac{k^{0.5}j^2 p}{m^{1.5}}\right)^{m-1} 
    \qquad \text{by (\ref{equ: k usefuler approx}).} \label{equ: x when i big}
\end{align}

We now split into three subcases:

\subsubsection{Subcase 1: $j < \frac{1}{k^2 p^3}$ and $m \leq 2j$}

Note that $j < \frac{1}{k^2 p^3} < \frac{1}{p}$, therefore we can use Lemma \ref{lem: y less than x} and also $e^{jp/2} = \Theta(1)$. We have, by Lemma \ref{lem: y less than x} and (\ref{equ: x when i big}):

\begin{align*}
     \sum_{j=1}^{\lfloor k^{-2} p^{-3} \rfloor} \sum_{m=1}^{2j} y_{k-j,m} &\leq  \sum_{j=1}^{\lfloor k^{-2} p^{-3} \rfloor} \sum_{m=1}^{2j} x_{k-j,m}*(jp)^{2j-m+1}\\&\leq
     \sum_{j=1}^{\lfloor k^{-2} p^{-3} \rfloor} \sum_{m=1}^{2j}   \left(\Theta(1)\frac{k^{2}}{npj^2} \right)^j \left(\Theta(1)\frac{k^{0.5}j^2 p}{m^{1.5}}\right)^{m-1}(jp)^{2j-m+1} \\&=
      \sum_{j=1}^{\lfloor k^{-2} p^{-3} \rfloor} \sum_{m=1}^{2j}  \left(\Theta(1)\frac{k^{2} p}{n} \right)^j \left(\Theta(1)\frac{k^{0.5}j}{m^{1.5}}\right)^{m-1}\\&\leq
      \sum_{j=1}^{\lfloor k^{-2} p^{-3} \rfloor} \sum_{m=1}^{2j}  \left(\Theta(1)\frac{k^{2} p }{n} \right)^j \left(\Theta(1)\frac{k^{0.5}}{m^{0.5}}\right)^{m-1}.
\end{align*}

Since, for $m < 2j < 2/p$, $(k/m)^{0.5(m-1)}$ increases by a multiplicative factor of $\Omega((kp)^{0.5}) \gg 1$ when $m$ increases by 1, the sum above is dominated by the case $m=2j$:

\begin{align*}
     \sum_{j=1}^{\lfloor k^{-2} p^{-3} \rfloor} \sum_{m=1}^{2j} y_{k-j,m} &= 
      \sum_{j=1}^{\lfloor k^{-2} p^{-3} \rfloor}   \left(\Theta(1)\frac{k^{2} p }{n} \right)^j \left(\Theta(1)\frac{k^{0.5}}{(2j)^{0.5}}\right)^{2j-1} \\&\leq
      \sum_{j=1}^{\lfloor k^{-2} p^{-3} \rfloor}   \left(\Theta(1)\frac{k^{3} p }{n} \right)^j \\&=
    o(1) \qquad \text{by (\ref{equ: k approx}).}
\end{align*}

We note that the last comparison is true only because $\frac{k^3p}{n} = o(1)$, and this is indeed the crucial bottleneck that requires $p \gg n^{-1/2} \ln^{3/2} n$ (strictly speaking, the case where $m = 2j+1$ is the tightest). \\

\subsubsection{Subcase 2: $j < \frac{1}{k^2 p^3}$ and $m > 2j$}

Again $e^{jp/2} = \Theta(1)$. By (\ref{equ: x when i big}),

\begin{align*}
\sum_{j=1}^{\lfloor k^{-2} p^{-3} \rfloor}\sum_{m=2j+1}^{k-j} x_{k-j,m} &\leq 
\sum_{j=1}^{\lfloor k^{-2} p^{-3} \rfloor}\sum_{m=2j+1}^{k-j}   \left(\Theta(1)\frac{k^{2} e^{jp/2}}{npj^2} \right)^j \left(\Theta(1)\frac{k^{0.5}j^2 p}{m^{1.5}}\right)^{m-1}\\&=
\sum_{j=1}^{\lfloor k^{-2} p^{-3} \rfloor}\sum_{m=2j+1}^{k-j}  \left(\Theta(1)\frac{k^{3}p j^2 }{n m^{3}} \right)^j \left(\Theta(1)\frac{k^{0.5}j^2 p}{m^{1.5}}\right)^{-2j+m-1} \\&\leq
\sum_{j=1}^{\lfloor k^{-2} p^{-3} \rfloor}\sum_{m=2j+1}^{k-j}  \left(\Theta(1)\frac{k^{3}p}{n m} \right)^j \left(\Theta(1)k^{0.5}j^{0.5} p\right)^{-2j+m-1} \\&\leq
\sum_{j=1}^{\lfloor k^{-2} p^{-3} \rfloor}\sum_{m=2j+1}^{k-j}  \left(\Theta(1)\frac{k^{3}p}{n} \right)^j \left(\Theta(1)(kp)^{-0.5}\right)^{-2j+m-1}.
\end{align*}

By (\ref{equ: k approx}), the sum over $m$ above is dominated by $m = 2j-1$:

\begin{align*}
\sum_{j=1}^{\lfloor k^{-2} p^{-3} \rfloor}\sum_{m=2j+1}^{k-j} x_{k-j,m} &\leq 
\sum_{j=1}^{\lfloor k^{-2} p^{-3} \rfloor}  \left(\Theta(1)\frac{k^{3}p}{n} \right)^j \\&=
o(1).
\end{align*}

\subsubsection{Subcase 3: $j \geq \frac{1}{k^2 p^3}$}

By (\ref{equ: k usefuler approx}) and (\ref{equ: x when i big}):

\begin{align*}
\sum_{j= \lceil k^{-2}p^{-3} \rceil}^{\lfloor k/2 \rfloor}\sum_{m=1}^{k-j} x_{k-j,m} &\leq \sum_{j= \lceil k^{-2}p^{-3} \rceil}^{\lfloor k/2 \rfloor}\sum_{m=1}^{k-j}   \left(\Theta(1)\frac{k^{2} e^{jp/2}}{npj^2} \right)^j \left(\Theta(1)\frac{k^{0.5}j^2 p}{m^{1.5}}\right)^{m-1} \\&\leq
\sum_{j= \lceil k^{-2}p^{-3} \rceil}^{\lfloor k/2 \rfloor}\sum_{m=1}^{k-j}   \left(\Theta(1)\frac{k^{2} e^{(k/2)p/2}}{np(k^{-2}p^{-3})^2} \right)^j \left(\Theta(1)\frac{k^{0.5}j^2 p}{m^{1.5}}\right)^{m-1}\\&=
\sum_{j= \lceil k^{-2}p^{-3} \rceil}^{\lfloor k/2 \rfloor}\sum_{m=1}^{k-j}   \left(\Theta(1)\frac{k^6 p^{5.5}}{n^{0.5}} \right)^j \left(\Theta(1)\frac{k^{1/3}j^{4/3} p^{2/3}}{m}\right)^{1.5(m-1)} 
\end{align*}

It is a straightforward calculus exercise to verify that, with $j$ fixed, the value of $m$ which maximizes a summand in the expression immediately above is $\Theta(k^{1/3} j^{4/3} p^{2/3})$, and hence, by (\ref{equ: k approx}):

\begin{align*}
\sum_{j= \lceil k^{-2}p^{-3} \rceil}^{\lfloor k/2 \rfloor}\sum_{m=1}^{k-j} x_{k-j,m} &\leq \sum_{j= \lceil k^{-2}p^{-3} \rceil}^{\lfloor k/2 \rfloor}k\left(\Theta(1)\frac{k^6 p^{5.5}}{n^{0.5}} \right)^j \exp\{\Theta(k^{1/3} j^{4/3} p^{2/3})\} \\&= 
\sum_{j= \lceil k^{-2}p^{-3} \rceil}^{\lfloor k/2 \rfloor}k \exp\{-\Theta(j\ln n)\} \exp\{\Theta(k^{1/3} j^{4/3} p^{2/3})\} \\&\leq
\sum_{j= \lceil k^{-2}p^{-3} \rceil}^{\lfloor k/2 \rfloor}\exp\{\ln n-\Theta(j\ln n) + \Theta(j\ln^{2/3} n)\} \\&=
o(1).
\end{align*}

This verifies (\ref{equ: sum of ys}), proving Lemma \ref{lem: 2nd moment}, therefore (ii) of Theorem \ref{thm: main} holds. 

\subsection{Closing remark}

Before closing this section, we remark how the range of $p$ is indeed tight. Consider the case where $k - \ell = j = 1$ and $m=3$. With a careful analysis of this Section as well as Lemma \ref{lem: hyperball}, one can show
that

\begin{equation*}
    y_{k-1,3} = \Omega \left(\frac{1}{\EE[Y]} \frac{k^3 p}{n}\right).
\end{equation*}

If $p = o(n^{-1/2} \ln^{3/2} n) \Longleftrightarrow \frac{k^3p}{n} = \omega(1)$, then Lemma \ref{lem: 2nd moment} no longer holds. \\

One could prove Lemma \ref{lem: 2nd moment} holds for slightly smaller $p$ if $Y_k$ instead counted induced $k$-trees where each vertex outside has at least 4 neighbors inside (instead of at least 3). However, if $p = o(n^{-1/2} \ln^{3/2} n) \Longleftrightarrow \frac{k^3p}{n} = \omega(1)$, then the expectation of this modified variable would {\it not} be asymptotically equal to $\EE[X_k]$. This ``drift" of expectation values will be further exploited in the next Section.

\section{Drift from the expectation threshold for small $p$} \label{sec: drift}

Here, we consider a third random variable. Let $W_k$ be the number of {\it maximal} induced trees with size $k$. An induced tree is a maximal induced tree if and only if no vertex outside has exactly one neighbor inside. Thus,

\begin{equation*}
    \frac{\EE[W_k]}{\EE[X_k]} = \left(1 - kp(1-p)^{k-1}\right)^{n-k}.
\end{equation*}

When $p = o(n^{-1/2})$, this ratio is small enough that it can be used to prove part (iii) of Theorem \ref{thm: main}.

\begin{proof}[Proof of (iii) of Theorem \ref{thm: main}]
    It is enough to show that whp $T(G_{n,p}) \not\in [k_0 - 1/(4np^2), k_0+1]$, since we already noted that $T(G_{n,p}) \leq k_0+1$ by Markov's Inequality using variable $X_{k_0+2}$ in Section \ref{sec: 1st moment}. Note that, for all $k \in [k_0 - 1/(4np^2), k_0+1]$, it satisfies (\ref{equ: k approx}) and (\ref{equ: k usefuler approx}).\\

    If $T(G_{n,p}) = k$, then certainly $W_k \geq 1$. Therefore, by Markov's Inequality, and using (\ref{equ: k approx}), (\ref{equ: k usefuler approx}), and (\ref{equ: ratio over k}):

    \begin{align*}
        &\Pr[T(G_{n,p}) \in [k_0 - 1/(4np^2), k_0+1]] \\&\leq \sum_{k = \lceil k_0 - 1/(4np^2) \rceil}^{k_0+1} \EE[W_k] \\&=
        \sum_{k = \lceil k_0 - 1/(4np^2) \rceil}^{k_0+1} \left(1 - kp(1-p)^{k-1}\right)^{n-k} \EE[X_k] \\&=
        \sum_{k = \lceil k_0 - 1/(4np^2) \rceil}^{k_0+1} \exp\{(-1+o(1))k_0 e^{-2}n^{-1} p^{-1}\} \EE[X_k] \\&=
        \exp\{(-1+o(1))k_0 e^{-2}n^{-1} p^{-1}\} \EE[X_{k_0+2}] \sum_{k = \lceil k_0 - 1/(4np^2) \rceil}^{k_0+1} \prod_{j=k}^{k_0+1} \frac{\EE[X_{j}]}{\EE[X_{j+1}]} \\&\leq
        \exp\{(-1+o(1))k_0 e^{-2}n^{-1} p^{-1}\} \left(3np\right)^{k_0 - (k_0 - 1/(4np^2)) + 2} \\&=
        \exp\{(-1+o(1))k_0 e^{-2}n^{-1} p^{-1} + (1+o(1))(k_0 p/2)/(4np^2) \} \\&=
        \exp\{(1/8-e^{-2}+o(1))k_0 n^{-1} p^{-1}\} \\&=
        o(1).
    \end{align*}
    
\end{proof}

\section{Closing remarks} \label{sec: closing remarks}

To prove two-point concentration for smaller $p$, one would need to consider a random variable that does not simply count induced trees of size $k$ with restrictions. Instead, the likely way forward would be to count sets with $k + \ell$ vertices, called ``clusters", which contain multiple induced trees of size $k$ inside. This kind of clustering technique was used by Bohman and the author \cite{BH-2point-Gnp} to prove two-point concentration of the independence number of $G_{n,p}$ for values of $p$ too small to just count independent sets. It appears that the ideal cluster structure for trees for $p$ slightly smaller than $n^{-1/2} \ln^{3/2} n$ is one whose 2-core, after contracting vertices of degree 2 (with respect to the $k$-core), is a random 3-regular multigraph, with a restriction for how many neighbors a vertex outside can have inside the cluster, as well as a restriction on what subset of the vertices of degree 3 in the 2-core can be cut vertices for the entire cluster. Even the first moment calculations needed here are surprisingly difficult. \\

One could also ask whether the above results can be modified to account for induced forests instead of induced trees. For constant $p$, the problem is simple because a significant proportion of the maximum induced forests are trees, such that one can perform second moment calculations on just trees and still derive 2-point concentration for forests; this is precisely what was done in \cite{KZ-forests-p-constant}. However, if $p = o(1)$, then it can be shown that the typical maximum induced forest of size $k$ contains a tree that covers all but roughly $1/p$ of the vertices in the forest. The remainder of the forest will behave like $G(1/p,p)$ conditioned on there being no cycles. Understanding the behavior of this remainder will likely involve a detailed analysis in the critical threshold of $G_{n,p}$ where a giant component emerges; specifically, where $p = \frac{1}{n} + O(n^{-4/3})$. This is required just to perform the first-moment calculations.

\vspace{1cm}

\appendix

\noindent{\large \bf Appendix: two technical lemmas} \\

The first lemma stated here is used for enumerating spanning trees covering a given forest with certain edge restrictions, proved in a more general form in \cite{AK-max-induced-forest}.

\begin{lem} \label{lem: joining forests}
    Suppose that $G$ is a forest with vertex set $[k]$, where $m$ of the components are trees labeled $F_1, \dots, F_m$, and the rest of the components are isolated vertices (a subtree $F_i$ could also be an isolated vertex). Let $f_i := |V(F_i)|$ and $\ell := \sum_{i=1}^{m} f_i$. Then the number of spanning trees $T$ on vertex set $[k]$ such that $G$ is a subgraph of $T$ and no edge in $T$ joins a subtree $F_i$ with another subtree $F_j$ is

\begin{equation*}
      \left(\prod_{i=1}^{m} f_i\right)k^{k-\ell-1}(k-\ell)^{m-1} .
\end{equation*}

\begin{proof}
    Use equation (16) from \cite{AK-max-induced-forest} with $h=1$ and the Binomial Theorem to derive the number of spanning trees $T$ but with a designated root; then divide by $k$ to get the result. \end{proof}

\end{lem}

The next lemma is used for bounding a sum that arises in the second moment calculations of Section \ref{sec: 2nd moment}. It uses several properties of Euler's Gamma function $\Gamma$, including (\ref{equ: Stirling's}) as well as two other properties which we state here (see (2.14) in \cite{Artin-gamma-funct} for (i) and \cite{Huber-volume-n-ball} for (ii)):

\begin{enumerate}[label=\roman*.]
    \item[(i)] $\Gamma(1/2) = \sqrt{\pi}$.

    \item[(ii)] Let $\mathcal{L}_{\RR^d}$ be the Lebesgue measure in $\RR^d$, and $\mathcal{B}(R)$ denote the ball with radius $R$. Then, for all $R > 0$ and $d \in \ZZ^+$:

    \begin{equation} \label{equ: ball hypervolume}
        \mathcal{L}_{\RR^d}(\mathcal{B}(R)) = \frac{(R \pi^{1/2})^{d}}{\Gamma(d/2+1)}.
    \end{equation}

\end{enumerate}

\begin{lem} \label{lem: hyperball}
    For all positive integers $\ell$ and $m$:

\begin{equation*}
\sum_{\substack{ f_1, \dots, f_m \in \ZZ^{+}:\\ f_1 + \dots + f_m = \ell}} \  \prod_{i=1}^{m}\frac{f_i^{f_i}}{f_i!} \leq \frac{e^{\ell} \ell^{(m-2)/2}}{2^{m/2} \Gamma(m/2)}.
\end{equation*}

\end{lem}

\begin{proof}

By (\ref{equ: Stirling's}),

\begin{equation} \label{equ: product to -1/2 power}
    \sum_{\substack{ f_1, \dots, f_m \in \ZZ^{+}:\\ f_1 + \dots + f_m = \ell}} \ \prod_{i=1}^{m}\frac{f_i^{f_i}}{f_i!}  \leq e^{\ell}(2 \pi)^{-m/2} \sum_{\substack{ f_1, \dots, f_m \in \ZZ^{+}:\\ f_1 + \dots + f_m = \ell}} \ \prod_{i=1}^{m}f_i^{-1/2}.
\end{equation}

If $m = 1$, then the right side of (\ref{equ: product to -1/2 power}) is just $e^{\ell} (2\pi)^{-1/2} \ell^{-1/2}$; since $\Gamma(1/2) = \sqrt{\pi}$ (see (i) above), we are done.\\

Next, we consider $m = 2$. Since $(j(\ell-j))^{-1/2}$ is decreasing in $j$ when $0 < j < \ell/2$ and increasing in $j$ when $\ell/2 < j < \ell$, then the above sum can be treated as a lower Riemann Sum for the corresponding integral between $0$ and $\ell$ (which can be solved with the substitution $x = \frac{\ell}{2} (1 + \sin \theta )$):

\begin{equation}
\sum_{\substack{ f_1, f_2 \in \ZZ^{+}:\\ f_1 + f_2 = \ell}} \ \prod_{i=1}^{2}f_i^{-1/2} = \sum_{j=1}^{\ell-1} (j(\ell-j))^{-1/2} \leq \int_{0}^{\ell} (x(\ell-x))^{-1/2} dx = \pi, \label{equ: pi integral}
\end{equation}

which resolves the case $m = 2$ using (\ref{equ: product to -1/2 power}) and $\Gamma(1) = 1$. Finally, let $m \geq 3$. Using (\ref{equ: pi integral}) again:

\begin{align*}
    \sum_{\substack{ f_1, \dots, f_m \in \ZZ^{+}:\\ f_1 + \dots + f_m = \ell}} \ \prod_{i=1}^{m}f_i^{-1/2} &= \sum_{\ell' = m-2}^{\ell-2} \left(\sum_{\substack{ f_1, \dots, f_{m-2} \in \ZZ^{+}:\\ f_1 + \dots + f_m = \ell'}}\prod_{i=1}^{m-2}f_i^{-1/2}\right)\left(\sum_{\substack{ f_{m-1},f_{m-2} \in \ZZ^{+}:\\ f_{m-1} + f_m = \ell - \ell'}}\prod_{i=m-1}^{m}f_i^{-1/2}\right) \\& \leq
    \pi \sum_{\ell' = m-2}^{\ell-2} \ \sum_{\substack{ f_1, \dots, f_{m-2} \in \ZZ^{+}:\\ f_1 + \dots + f_m = \ell'}} \ \prod_{i=1}^{m-2}f_i^{-1/2} \\&\leq
    \pi \sum_{\substack{ f_1, \dots, f_{m-2} \in \ZZ^{+}:\\ f_1 + \dots + f_m \leq \ell}} \ \prod_{i=1}^{m-2}f_i^{-1/2}.
\end{align*}

The expression $\prod_{i=1}^{m-2}f_i^{-1/2}$ is decreasing with respect to all $f_i$, therefore the final sum is a lower Riemann sum for the following $m-2$-dimensional integral:

\begin{align*}
\pi \sum_{\substack{ f_1, \dots, f_{m-2} \in \ZZ^{+}:\\ f_1 + \dots + f_m \leq \ell}} \ \prod_{i=1}^{m-2}f_i^{-1/2} \leq &\pi 
\int\limits_{\substack{x_i \geq 0 \\
    x_1 + \dots + x_{m-2} \leq \ell}} \left(\prod_{i=1}^{m-2} x_i^{-1/2}\right) d x_{m-2}\dots dx_{1} \\= & \pi \int\limits_{\substack{y_i \geq 0 \\
    y_1^2 + \dots + y_{m-2}^2 \leq \ell}} 2^{m-2} d y_{m-2}\dots dy_{1} \qquad \text{(substituting $x_i$ with $y_i^2$)} \\= & \pi \int\limits_{\substack{
    y_1^2 + \dots + y_{m-2}^2 \leq \ell}} 1 d y_{m-2}\dots dy_{1}.
\end{align*}

The region of the last integral is simply a ball with radius $\sqrt{\ell}$ in $\RR^{m-2}$, so the last expression equals $\pi \cdot \frac{(\ell\pi)^{(m-2)/2}}{\Gamma(m/2)}$ by (\ref{equ: ball hypervolume}). Therefore, by (\ref{equ: product to -1/2 power}):

\begin{equation*}
    \sum_{\substack{ f_1, \dots, f_m \in \ZZ^{+}:\\ f_1 + \dots + f_m = \ell}} \ \prod_{i=1}^{m}\frac{f_i^{f_i}}{f_i!} \leq 
    e^{\ell}(2 \pi)^{-m/2} \pi  \left(\frac{(\ell\pi)^{(m-2)/2}}{\Gamma(m/2)}\right) =
    \frac{e^{\ell} \ell^{(m-2)/2}}{2^{m/2} \Gamma(m/2)}. 
\end{equation*}
\end{proof}

\end{document}